\documentclass{elsarticle}

\usepackage{amssymb, latexsym, comment, amsmath, amsthm, upgreek, epsf, epsfig, color, verbatim, caption, tabularx}
\usepackage{amsfonts}
\usepackage{booktabs}
\usepackage{siunitx}

\usepackage{array}
\newcolumntype{L}[1]{>{\raggedright\let\newline\\\arraybackslash\hspace{0pt}}m{#1}}
\newcolumntype{C}[1]{>{\centering\let\newline\\\arraybackslash\hspace{0pt}}m{#1}}
\newcolumntype{R}[1]{>{\raggedleft\let\newline\\\arraybackslash\hspace{0pt}}m{#1}}

\makeatletter
\def\ps@pprintTitle{%
\let\@oddhead\@empty
\let\@evenhead\@empty
\def\@oddfoot{\centerline{\thepage}}%
\let\@evenfoot\@oddfoot}
\makeatother

\newtheorem{thm}{Theorem}[section]
\newtheorem{lemma}[thm]{Lemma}
\newtheorem{prop}[thm]{Proposition}
\newtheorem{cor}[thm]{Corollary}
\newtheorem{defn}[thm]{Definition}
\newtheorem{rem}[thm]{Remark}
\newtheorem{example}[thm]{Example}

\newcommand{\bs}{\boldsymbol}
\begin{document}

\begin{frontmatter}
\title{Leaf Space Isometries of Singular Riemannian Foliations and Their Spectral Properties}
\author{Ian M.~Adelstein \\ M.~R.~Sandoval} 
\address{Department of Mathematics, Yale University\\ New Haven, CT 06520 United States\\Department of Mathematics, Trinity College\\ Hartford, CT 06106 United States}
\begin{abstract} In this note, the authors show by example that an isometry between leaf spaces of singular Riemannian foliations need not induce an equality of the basic spectra. If the leaf space isometry preserves the mean curvature vector fields, then it is proved that the basic spectra are equivalent,\,i.e.\,that the leaf spaces are isospectral. As a corollary to the main result, the authors identify geometric conditions that ensure preservation of the mean curvature vector fields, and therefore isospectrality of the leaf spaces.

\end{abstract}
\begin{keyword} Spectral geometry, Laplace operator, orbifolds, orbit spaces, group actions
\MSC[2010] 58J50 \sep 58J53 \sep 22D99 \sep 53C12
\end{keyword}
\end{frontmatter}

\section*{Acknowledgements}
The authors would like to thank Carolyn Gordon for many helpful conversations throughout the course of this project, as well as Emilio Lauret and Marco Radeschi for providing valuable feedback. The authors would also like to thank the reviewers for many helpful comments, including suggesting a much shorter, more elegant proof of the main corollary, and to acknowledge the support of the National Science Foundation, Grant DMS-1632786.

\begin{section}{Introduction} 

Given a singular Riemannian foliation with closed leaves on a  compact Riemannian manifold $M$, we consider the spectral geometry of the leaf space $M/\mathcal{F}$. More precisely, we are interested in the $\mathcal{F}$-basic spectrum--the spectrum of the Laplacian on $M$ restricted to the smooth functions that are constant on the leaves of the foliation. Let $C^{\infty}_B(M,\mathcal{F})$ denote the space of such functions. The leaf space $M/\mathcal{F},$ which may be quite singular, has a natural ``smooth" structure given by the algebra $C^{\infty}(M/\mathcal{F})$ consisting of functions $f \colon M/\mathcal{F} \to \mathbb{R}$ whose pullback via $\pi \colon M \to M/\mathcal{F}$ are smooth $\mathcal{F}$-basic functions on $M$, i.e.~$f \in C^{\infty}(M/\mathcal{F})$ if and only if $\pi^* f \in C^{\infty}_B(M,\mathcal{F})$. (See Section 2 for a more detailed definition.) A map $\varphi \colon M_1/\mathcal{F}_1 \to M_2/\mathcal{F}_2$ is said to be \emph{smooth} if the pullback of every smooth function on $M_2/\mathcal{F}_2$ is a smooth function on $M_1/\mathcal{F}_1$, i.e.~$\varphi^*f  \in C^{\infty}_B(M_1/\mathcal{F}_1)$ for every $f \in C^{\infty}(M_2/\mathcal{F}_2)$. The following notion of smooth isometry of leaf spaces using this idea of smoothness has recently appeared in the literature relating to singular Riemannian foliations (SRF):

\begin{defn}\label{d:srfiso} A map $\varphi \colon M_1/\mathcal{F}_1 \to M_2/\mathcal{F}_2$ is said to be a {\bf smooth SRF isometry} if it is an isometry of metric spaces that is smooth (in the above sense) with smooth inverse. 
\end{defn}

Given the above notions of a smooth SRF isometry between leaf spaces and ``smooth" structures on leaf spaces, it is natural to ask if these generalizations to possibly quite singular leaf spaces have the same properties as the corresponding structures on smooth manifolds. In particular, one can ask if the existence of such a smooth SRF isometry between leaf spaces $M_1/\mathcal{F}_1$ and $M_2/\mathcal{F}_2$ implies equivalence of the $\mathcal{F}_i$-basic spectra for $i=1, 2$. As we shall see in the following example, the existence of a smooth SRF isometry of the leaf space is not sufficient to guarantee that $spec_B(M_1,\mathcal{F}_1)=spec_B(M_2,\mathcal{F}_2)$.
\begin{example}\label{ex1} \normalfont Let $M= S^2$ be the round 2-sphere with $G=SO(2) \leq \text{Isom}(M)$ where $G$ acts by rotation around an axis. The orbit space of such an action defines a singular Riemannian foliation with the orbits as leaves. By Theorem 1.3 of \cite{AL2011}, the identity map on the orbit space is a smooth SRF isometry between the quotient space $M/G$ and the orbifold $\mathcal{O}=[0, \pi]$ given by the standard association of orbits to points. The Neumann (orbifold) spectrum of $\mathcal{O}$ is known to be $\text{spec}(\mathcal{O}) = \{  0, 1, 4, \ldots, k^2, \ldots \}$. The eigenvalues of the $n$-sphere are given by $k(k+n-1)$ so that the $G$-invariant spectrum of $M$ is a subset of $\text{spec}(M) = \{ 0, 2, 6, \ldots, k(k+1) , \ldots \}$ where we have suppressed the multiplicity of eigenvalues. It follows immediately that the smooth SRF isometry between $M/G$ and $\mathcal{O}=[0, \pi]$ does not induce an equality of the basic spectra.
\end{example}

We note that if the leaf spaces $M_1/\mathcal{F}_1$ and $M_2/\mathcal{F}_2$ have the structure of $n$-dimensional orbifolds\footnote{Here, a leaf space $M/\mathcal{F}$ has the structure of an orbifold if for every $p\in M,$ there exists an open $U$ neighborhood of $p$ such that the local quotient $U/\mathcal{F}$ is a Riemannian orbifold. This is equivalent to $(M,\mathcal{F})$ being {\it infinitesimally polar}, see Theorem 1.4 of \cite{LT2010} for more detail.}, $\mathcal{O}_1$ and $\mathcal{O}_2,$ respectively, then we have two different representations of each leaf space: the frame bundle representation $\mathcal{O}_i=\mathcal{F}r(\mathcal{O}_i)/O(n)$, and the original one, $\mathcal{O}_i=M_i/\mathcal{F}_i$, for $i=1,\,  2$. The frame bundle representation is rather special because any orbifold isometry, $\varphi:\mathcal{O}_1\rightarrow\mathcal{O}_2$ automatically lifts to an isometry of the corresponding orthonormal frame bundles $\mathcal{F}r(\mathcal{O}_1)\rightarrow\mathcal{F}r(\mathcal{O}_2),$ as described in Section 4 of \cite{BZ2007}. Furthermore, this isometry of the frame bundles is actually a foliated diffeormorphism between the frame bundles, i.e.~it sends the leaves to leaves (which are the orbits under the natural $O(n)$ action). The existence of such a foliated diffeomorphism on the frame bundles is sufficient to guarantee equality of the orbifold spectra. 

In the more general setting of arbitrary quotients, no such foliated diffeomorphism between $M_1$ and $M_2$ is guaranteed by the existence of a smooth SRF isometry $\varphi:M_1/\mathcal{F}_1\rightarrow  M_2/\mathcal{F}_2$. Indeed, the question of the existence of such a foliated isometry is highly non-trivial. If one were to seek to generalize the situation with the frame bundle representation to more general leaf spaces of singular Riemannian foliations then the analogous assumption of the existence of a leaf-preserving isometry between the ambient spaces $M_1$ and $M_2$ is certainly sufficient but overly strong. In this paper, we explore conditions under which we can guarantee that $spec_B(M_1,\mathcal{F}_1)=spec_B(M_2,\mathcal{F}_2).$ The main result of this paper shows that, given a smooth SRF isometry between leaf spaces, $\mathcal{F}_i$-basic isospectrality can be assured under less stringent conditions in special cases.

Recall the following definition from the theory of singular Riemannian foliations: a vector field $X$ on $(M,\mathcal{F})$ is said to be {\it basic} if $X_p$ takes values in the normal space to the leaves, $\nu_p L_p,$ for all $p$ in its domain, and if, in local distinguished coordinates, the coefficient functions are basic functions. Such functions project in a well-defined way to the quotient leaf space, $M/\mathcal{F}$. We denote this projection by $X_*$. In what follows, we will be particularly interested in the mean curvature vector field of the leaves, defined over the region of $(M,\mathcal{F})$ where the leaves are of maximal dimension. This region is known as the regular region, and it is open and dense in $M$. If $H$ is basic, then $(M,\mathcal{F})$ is said to be a {\it generalized isoparametric} singular Riemannian foliation. (See, for example, \cite{Pac} or \cite{AR2016b} for definitions and results relating to basic mean curvature vector fields.)

\begin{thm}\label{main} Let $(M_1, \mathcal{F}_1)$ and $(M_2, \mathcal{F}_2)$ be two singular Riemannian foliations with mean curvature vector fields $H_1$ and $H_2$, respectively, and let $\varphi \colon M_1/\mathcal{F}_1 \to M_2/\mathcal{F}_2$ be a smooth SRF isometry satisfying the following two conditions:  (1) $H_1$ and $H_2$ are basic, and  (2) $d\varphi(H_{1*})=H_{2*}.$ Then the leaf spaces are isospectral, i.e.~$spec_B(M_1,\mathcal{F}_1)=spec_B(M_2,\mathcal{F}_2).$ 
\end{thm}

\begin{rem}
\normalfont If either representation has only minimal leaves, then so must the other, and by the result of \cite{MiWol2006}, the representations of the leaf spaces must arise from regular Riemannian foliations. In that case the leaf space is either a manifold or an orbifold, and the associated basic spectrum is the spectrum of the Laplacian on functions on the quotient manifold/orbifold. 
\end{rem}

\begin{rem}
\normalfont In addition to the minimal case, the hypotheses of the main theorem can be shown to be satisfied by the isometry constructed by K. Richardson in \cite{Ri2} between the leaf space of an SRF defined by the closure of a regular Riemannian foliation with irregular closure and the orbit space of an isometric group action on a related manifold, defined via a suspension.
\end{rem}
The proof of Theorem \ref{main} is rather straightforward. Our main results also include the following corollary which yields isospectrality in a special case that guarantees condition (2) above.

\begin{cor}\label{cor:spaceforms}
If $M_1$ and $M_2$ are compact space forms of the same non-negative curvature $\kappa$ that admit singular Riemannian foliations $\mathcal{F}_1$ and $\mathcal{F}_2,$ respectively, and $\varphi:M_1/\mathcal{F}_1 \rightarrow M_2/\mathcal{F}_2$ is a smooth SRF isometry that preserves the codimensions of corresponding leaves, then $spec_B(M_1,\mathcal{F}_1)=spec_B(M_2,\mathcal{F}_2).$ 

\end{cor}

\begin{rem}
\normalfont The corollary above implies the isospectrality result for compact manifolds of constant non-negative curvature. In the case of ambient manifolds of negative curvature, constant or otherwise, it is known that such manifolds admit no non-trivial Riemannian foliations, either regular or singular (due to the work of Zhegib in \cite{Zhegib}, and Lytchak in \cite{Lytchak}, respectively).  This accounts for all possible compact space forms.
\end{rem}

We also have the following, in the case the leaf space has no topological boundary.

\begin{rem}
\normalfont In Corollary \ref{cor:spaceforms}, if the quotient space has no topological boundary in the sense of \cite{AR2015}, then the condition on the preservation of leaf codimensions can be dropped. As discussed in Section 2.1 of \cite{AR2015}, the condition that $M/\mathcal{F}$ has no boundary is equivalent to the condition that the quotient codimension\footnote{See Section 3.2 for definition.} of every singular stratum is greater than one. 

\end{rem}

Finally, we have two further corollaries of the main theorem, which relate to the special case of isometric group actions. These are stated and proved in Section 3.

\begin{rem}
\normalfont If the orbit spaces $M_1/\mathcal{F}_1=:\mathcal{O}_1$ and $M_2/\mathcal{F}_2=:\mathcal{O}_2$ have the structure of Riemannian orbifolds, then Definition \ref{d:srfiso} corresponds to the usual notion of a smooth isometry between Riemannian orbifolds in the sense of preserving the sheaf of smooth functions on the orbifold. Thus, if $\varphi: M_1/\mathcal{F} _1\rightarrow M_2/\mathcal{F}_2$ is a smooth SRF isometry satisfying the hypotheses of Theorem \ref{main}, then $spec_B(M_1,\mathcal{F}_1)=spec_B(M_2,\mathcal{F}_2).$ The $O(n)$-invariant spectra of the frame bundles $\mathcal{F}r(\mathcal{O}_1)$ and $\mathcal{F}r(\mathcal{O}_2)$ will also be the same, by virtue of the fact that such isometries lift to the associated frame bundles, as noted earlier. However, unless there is an additional isometry between $M_i/\mathcal{F}_i$ and $\mathcal{F}r(\mathcal{O}_i)/O(n)$ for at least one of $i=1,2$ that satisfies the hypotheses of Theorem \ref{main} or its corollaries, then there is no guarantee that the $\mathcal{F}_i$-basic spectra are equal to the orbifold spectra of $\mathcal{O}_i$.
\end{rem}

The main results follow similar work by M. Alexandrino and M. Radeschi in \cite{AR2015}. In particular, we note that some nice consequences follow from Theorem 1.1 of \cite{AR2015}, which states that a metric space isometry between leaf spaces that preserves the leaf codimensions must be smooth. Thus, one may drop the adjective ``smooth" from any SRF isometry that preserves the leaf codimensions from Corollary \ref{cor:spaceforms}. In what follows, an SRF isometry satisfying this hypothesis will be understood to be smooth.

The paper proceeds as follows. We discuss in more detail the notion of a smooth SRF isometry and its properties in Section~\ref{smooth}. We also prove Theorem~\ref{main} and discuss the obstacle to isospectrality when a smooth SRF isometry exists between leaf spaces. Finally, we prove the main corollary concerning space forms.  In Section~\ref{consequences}, we discuss applications of Theorem \ref{main} to isometric group actions. We also make note of some properties of SRF isometries. In the Appendix, we have also included an alternative proof of Corollary \ref{cor:spaceforms} using eigenvalues of the shape operator, Jacobi fields, and some special properties of foliations on spheres.

\end{section}

\begin{section}{Singular Riemannian Foliations and Smooth Isometries}\label{smooth}
\begin{subsection}{Preliminaries}
We recall some of the terminology of singular Riemannian foliations, \cite{Mol}.
\begin{defn}\label{SRF}
A {\it singular Riemannian foliation} on an ambient manifold $M$ is a partition $\mathcal{F}$ of $M$ by connected immersed submanifolds, known as the leaves, that satisfy the following two conditions:
\begin{enumerate}
\item The module $\Xi_{\mathcal{F}}$ of smooth vector fields that are tangent to the leaves is transitive on each leaf in the sense that there exist a collection of smooth vector fields $X_i$ on $M$ such that for each $x\in M$ the tangent space to the leaf $L_x$ through $x$ is spanned by the vectors $X_i$. Note that the dimension of the leaves may vary over the manifold.
\item There exists a Riemannian metric $g$ on $M$ that is adapted to $\mathcal{F}$ in the sense that every geodesic that is perpendicular at one point to a leaf remains perpendicular to every leaf that intersects that geodesic. In other words, the normal distribution to the leaves is totally geodesic.
\end{enumerate}
\end{defn}
\noindent Examples: 
\begin{enumerate}
\item Given a compact Lie group acting on a compact manifold $M$ by isometries, the partition of $M$ into orbits defines a singular Riemannian foliation whose leaf space is the orbit space of the action. Furthermore, these singular Riemannian foliations are generalized isoparametric foliations, by a result of Pacini, \cite{Pac}, and thus satisfy condition (1) of the main theorem.
\item As a special case of the above, an orbifold $\mathcal{O}$ of dimension $n$ can always be represented as the leaf space of the orthonormal frame bundle $\mathcal{F}r(\mathcal{O})$ with the leaves being the orbits under the natural $O(n)$ action. Because the isotropy subgroups are finite, this singular Riemannian foliation is actually a {\it regular Riemannian foliation} in that all the leaves (orbits) are of the same dimension.
\item Given a Riemannian foliation $(M,\mathcal{F})$, the partition of $M$ into leaf closures defines a singular Riemannian foliation when the dimension of leaf closures is not necessarily constant.
\end{enumerate}

In the above situations, the examples produce (possibly) singular Riemannian foliations with closed leaves. In general, there is no requirement that a singular Riemannian foliation should have closed leaves. However, when the leaves are closed, one can define a notion of a smooth structure on $M/\mathcal{F}$ as well as a natural metric structure that is inherited from $M$. Note that it has recently been proved in \cite{AR2017} that the closure of a singular Riemannian foliation is again a singular Riemannian foliation, settling a long-standing conjecture due to Molino. Hence, if one has a singular Riemannian foliation whose leaves are not all closed one may always consider its closure instead.

\begin{defn}\label{d:smoothiso}
We define a ``smooth" structure on $M/\mathcal{F}$ to be the algebra $C^\infty(M/\mathcal{F})$ consisting of functions $f: M/\mathcal{F}\rightarrow \mathbb{R}$ whose pullback via $\pi: M\rightarrow M/\mathcal{F}$ is a smooth basic function on $M.$ Note that the ``smooth" structure of a possibly non-closed leaf space $M/\mathcal{F}$ will be the same as that of $M/\overline{\mathcal{F}}.$ In the situation when the singular Riemannian foliation arises from an isometric group action by a compact Lie group, we can rephrase this as follows.  A smooth structure on $M/G$ is the algebra of functions $C^\infty(M/G)$ consisting of functions $f: M/G\rightarrow \mathbb{R}$ whose pullback via $\pi: M\rightarrow M/G$ is a smooth $G$-invariant function on $M$. Thus, if $U$ is an open set in the topological space $M/G,$ with the quotient topology,  then smooth structure is given by the sheaf of smooth functions $C^\infty(U):=(C^\infty(\pi^{-1}(U)))^G.$
\end{defn}

We also recall the notion of singular Riemannian foliations with disconnected leaves, from Section 2 of \cite{AR2015}. At times in the study of singular Riemannian foliations, it is natural to consider singular Riemannian foliations with disconnected leaves. For example, the lift of a foliation $(M,\mathcal{F})$ by a surjective map $p:N\rightarrow M$ produces a foliation whose leaves are the pre-images of the leaves of $\mathcal{F}$, but these pre-images need not be connected. One can define the lifted foliation as the connected components of the pre-images, but this is not always desirable. A singular Riemannian foliation with disconnected leaves is defined as the triple $(M,\mathcal{F}^0, \Gamma)$ where $(M,\mathcal{F}^0)$ is a singular Riemannian foliation with connected leaves, and $\Gamma$ is a discrete group of isometries of the leaf space $M/\mathcal{F}^0$. The action of $\Gamma$ on the leaf space extends naturally to an action on the leaves of $\mathcal{F}^0,$ with the disconnected leaves defined to be the orbits under this action, $\Gamma\cdot L_p$ for $L_p\in \mathcal{F}^0$. The notions of isometry between the leaf spaces of singular Riemannian foliations and the mean curvature vector field extend to singular Riemannian foliations with disconnected leaves. However, some care should be taken with the definition of the principal leaves. A leaf $L_p$ of a disconnected foliation is said to be a {\it principal leaf} if it satisfies the following two conditions:  (1) each connected compeonent of $L_p$ is a principal leaf of $\mathcal{F}^0$ in the usual sense, i.e. the dimension of $L_p$ is maximal, and it has trivial holonomy; and (2) if there exists an isometry $\gamma \in\Gamma$ that fixes any component of $L_p$, then $\gamma$ must be the identity in $\Gamma.$


In what follows, we will occasionally make use of the natural stratification of the ambient space $M$. Recall from the theory of Riemannian foliations, \cite{Mol}, that one can define a stratification of $M$ as follows. Let $\Sigma_k$ denote the union of leaves of dimension $k$. Then for each $k$, $\Sigma_k$ is a weakly embedded submanifold. When $k$ is maximal, the corresponding stratum is usually denoted by $M_{reg}$, which is an open, dense subset of $M$, whose quotient $M_{reg}/\mathcal{F}$ is at worst a Riemannian orbifold. 
\end{subsection}

\begin{subsection}{The proofs of the main results}\label{proofs}
We first review some of the local geometry of singular Riemannian foliations. While the leaf space of a singular Riemannian foliation is an example of an Aleksandrov space, the metric structure is much stronger than the more general case of an Aleksandrov space because the metric structure on the quotient comes from a smooth metric structure on the manifold $M$ that is adapted to the foliation, (see \cite{Mol} for definitions). In fact, the distance between points in the quotient is realized by the lengths of orthogonal geodesics connecting the leaves containing the preimages of the points in the quotient. The set $M_{reg}$ admits a regular Riemannian foliation, and thus a transverse metric $g^T,$ such that $\pi |_{M_{reg}}: M_{reg}\rightarrow M_{reg}/\mathcal{F}$ has as its image a Riemannian orbifold $B$ whose metric $g_B$ is isometric to $g^T$ in the sense that $g^T=\pi^*g_B$. This induces the following relationship between the local Laplacian on $M$ and the Laplacian on the quotient orbifold $B$ in terms of the mean curvature vector field $H$ (or equivalently, its dual, the mean curvature form) over a neighborhood $U$ contained within the regular region (see standard references such as \cite{GW2009}):
\begin{equation}\label{laplacians}
\Delta_{U}f=\Delta_{B}f-g(\nabla f , H_*),
\end{equation}
where $f$ is a smooth basic function, and thus defines a function on $B$ that we also denote by $f$. We note that if $f$ is basic and the mean curvature vector field over the regular region is basic, then the last term in the right-hand side of  \eqref{laplacians} also defines a basic function.

\begin{proof}[Proof of Theorem \ref{main}.]

Recalling that the union of the leaves of maximal dimension is the regular part of $M_i$ for $i=1, 2$, let $p_1$ (respectively $p_2$) be in the regular part of $M_1$, (respectively $M_2$). The corresponding leaf spaces are at worst orbifolds and let  $B_i$ denote the local model for these orbifolds $(M_i)_{reg}/\mathcal{F}_i$ for $i=1, 2$.  Suppose for each $i=1,\,2$ there is a neighborhood $U_i$ of $p_i$ such that $\pi_i(U_i)\subset (M_i)_{reg}/\mathcal{F}_i.$ By shrinking these neighborhoods if necessary, we may assume that $\varphi(\pi_1(p_1))=\pi_2(p_2)$ and also that $\varphi(\pi_1(U_1))=\pi_2(U_2).$ Note that $\varphi$ restricts to an isometry from $\pi_1(U_1)$ to $\pi_2(U_2)$.

For each $i=1,\,2$
\begin{align}
\Delta_{U_1}f_1&=\Delta_{B_1}f_1-g_1(\nabla f_1, H_{1*}),\label{ub1}\\
\Delta_{U_2}f_2&=\Delta_{B_2}f_2-g_2(\nabla f_2, H_{2*}),\label{ub2}
\end{align}
for $f_1$ an $\mathcal{F}_1$-basic function which defines a function in $C^\infty(B_1),$ also denoted by $f_1$, and similarly for  $f_2$ an $\mathcal{F}_2$-basic function and its corresponding function in $C^\infty(B_2),$ as in \eqref{laplacians}. Let $\Delta _1$ denote the operator on $\pi_1\bigl((M_1)_{reg}\bigr)$ on the right-hand side of \eqref{ub1}, and let $\Delta_2$ be defined similarly for the right-hand side of \eqref{ub2}.

Let $f_1=\varphi^*f_2.$ As noted in the proof of Propositon 3.5 of \cite{AR2015}, because $\varphi$ lifts to an isometry which we also call $\varphi$ from $B_1$ to $B_2$, we have
\begin{equation}\label{orbisom}
\varphi^*\Delta_{B_2}f_2=\Delta_{B_1}\varphi^*f_2=\Delta_{B_1}f_1.
\end{equation}

Since $d\pi_2(H_2)=H_{2*}$ is basic, every function in \eqref{ub2} defines a function on the quotient, so we may pull them back via $\varphi.$ Applying $\varphi^*$ to both sides of \eqref{ub2}, and using \eqref{orbisom} yields:
\begin{align}
\varphi^*\Delta_{U_2}f_2&=\varphi^*\Delta_{B_2}f_2-\varphi^*\bigl(g_2(\nabla f_2, H_{2*})\bigr)\\
&=\Delta_{B_1}\varphi^*f_2-\varphi^*\bigl(g_2(\nabla f_2, H_{2*})\bigr).
\end{align}
But then observe that since $\varphi_*(H_{1*})=H_{2*},$ and $\varphi^*g_2=g_1$ we have
\begin{equation}
\varphi^*\bigl(g_2(\nabla f_2, H_{2*})\bigr)=g_1(\nabla f_1, H_{1*}),
\end{equation} 
and hence, comparing the above to \eqref{ub1}, we have
\begin{equation}\label{e:b2}
\varphi^*\Delta_{U_2}f_2=\Delta_{U_1}f_1=\Delta_{U_1}\varphi^*f_2,
\end{equation}
which is the local form of the desired intertwining $\varphi^*\Delta_{2}=\Delta_{1}\varphi^*.$

Now suppose $h\in C^\infty(M_2/\mathcal{F}_2)$ is such that its restriction to $\pi_2\bigl((M_2)_{reg}\bigr)$ is an eigenfunction of $\Delta_{2}$ with eigenvalue $\lambda.$  Then we have via \eqref{e:b2},
\begin{equation}
\Delta_{M_2}\pi_2^*h=\pi_2^*\Delta_{2}h=\lambda \pi_2^*h \text { on } (M_2)_{reg}.
\end{equation}
By continuity $\pi_2^*h$ is a $\mathcal{F}_2$-invariant eigenfunction of $\Delta_{M_2}$ with eigenvalue $\lambda$. Conversely, every $f_2\in C_B^\infty(M_2,\mathcal{F}_2)$ that is an eigenfunction of $\Delta_{M_2}$ with eigenvalue $\lambda$ descends to a local eigenfunction $h_2$ of $\Delta_{2}$ with corresponding eigenvalue $\lambda$.

By \eqref{e:b2}, we now have
\begin{equation}\label{e:evalue}
\lambda \pi_1^*\varphi^*h=\Delta_{M_1}\pi_1^*\varphi^*h \text{ on } (M_1)_{reg}.
\end{equation}
But $\pi_1^*\varphi^*h$ is a smooth $\mathcal{F}_1$-invariant function on all of $M_1$, as is $\Delta_{M_1}\pi_1^*\varphi^*h$. Hence, by continuity, \eqref{e:evalue} holds on all of $M_1$, and $spec(M_2,\mathcal{F}_2)\subset spec(M_1,\mathcal{F}_1).$ Finally, since $\varphi$ has a smooth inverse, the reverse inclusion holds as well.
\end{proof}

To prove Corollary \ref{cor:spaceforms}, we first state and prove two lemmas. The first lemma is a generalization of Proposition 3.1 of \cite{AR2015} to spheres. The second lemma shows that the mean curvature vector field is preserved by a covering map.

\begin{lemma}\label{lemmaA}
Let $\mathbb{S}^{n_i}_{r}$ denote a sphere of radius $r$ and dimension $n_i$, for $i=1,2.$ Let $(\mathbb{S}^{n_1}_{r},\mathcal{F}_1)$ and $(\mathbb{S}^{n_2}_{r},\mathcal{F}_2)$ be two (possibly disconnected) singular Riemannian foliations with closed leaves, and let $\varphi: \mathbb{S}^{n_1}_{r}/\mathcal{F}_1\rightarrow \mathbb{S}^{n_2}_{r}/\mathcal{F}_2$ be an isometry that preserves the codimensions of the leaves. Then the mean curvature vector fields of the corresponding principal leaves are basic, and $\varphi$ preserves the projections of those vector fields.
\end{lemma}


\begin{proof}
First observe that a singular Riemannian foliation of a sphere $(\mathbb{S}^{n_i}_{r}, \mathcal{F}_i)$ can be extended to a singular Riemannian foliation $\mathcal{F}'_i$ on $\mathbb{R}^{n_i+1}$ whose leaves consist of the origin and the images of the leaves of $\mathcal{F}_i$ under homothetic transformations. In spherical coordinates $(\rho, \theta)$ on $\mathbb{R}^{n_i+1}$ with $\theta\in \mathbb{S}^{n_i}_{r},$ the leaves of $\mathcal{F}'_i$ are the radial extensions of the leaves of $\mathcal{F}_i$, and the restriction of the leaves of $\mathcal{F}'_i$ to $\mathbb{S}^{n_i}_{r}\subset \mathbb{R}^{n_i+1}$ is accomplished by setting $\rho=r$. Thus, the restriction of the leaves of $\mathcal{F}'_i$ to $\mathbb{S}^{n_i}_{r}\subset \mathbb{R}^{n_i+1}$ yields the original foliation $(\mathbb{S}^{n_i}_{r}, \mathcal{F}_i)$.

Let $H_i$ and $H_i'$ denote the mean curvature vector fields of $(\mathbb{S}^{n_i}_{r}, \mathcal{F}_i)$ and $(\mathbb{R}^{n_i+1},\mathcal{F}'_i),$ respectively. As noted in \cite{AR2016b}, any singular Riemannian foliation in spheres or Euclidean space is generalized isoparametric, and thus $H_i$ and $H'_i$ are basic, for $i=1,2$.

Next, note that the isometry $\varphi:\mathbb{S}^{n_1}_{r}/\mathcal{F}_1\rightarrow \mathbb{S}^{n_2}_{r}/\mathcal{F}_2$ induces an isometry $\varphi^0:\mathbb{R}^{n_1+1}/\mathcal{F}'_1\rightarrow \mathbb{R}^{n_2+1}/\mathcal{F}'_2.$ Proposition 3.1 of \cite{AR2015} now can be applied to $\varphi^0$ to show that it preserves the projections of the basic vector fields of $(\mathbb{R}^{n_1+1},\mathcal{F}'_1)$ and $(\mathbb{R}^{n_2+1},\mathcal{F}'_2).$ Since the foliations are generalized isomparametric SRFs, we may apply Lemma 5.2 of \cite{AR2016b}, to express $H'_i=-\nabla (log(Lvol(x))$ where $Lvol(x)$ denotes the volume of the leaf $L_x$ of $\mathcal{F}'_i$. By expressing $H'_i$ in spherical coordinates, $(\rho,\theta)$, $\theta\in \mathbb{S}^{n_i}_{r}$ on $\mathbb{R}^{n_i+1}$ one sees that $H'_i$ and $H_i$ have the same projections onto their respective leaf spaces for $i=1, 2$. It now follows that $\varphi: \mathbb{S}^{n_1}_{r}/\mathcal{F}_1\rightarrow \mathbb{S}^{n_2}_{r}/\mathcal{F}_2$ preserve projections of the mean curvature vector fields $H_1$ and $H_2$.\end{proof}

\begin{lemma}\label{cor:coveringspace}
Suppose $\tilde{p}: \tilde{M}\rightarrow M$ is a covering map and $(M,\mathcal{F})$ is a generalized isoparametric singular Riemannian foliation. Then the lift of $\mathcal{F}$ by $\tilde{p}$, which we denote by $\tilde{\mathcal{F}},$ is a singular Riemannian foliation of $\tilde{M}$ (with possibly disconnected leaves) with respect to the lifted metric. Furthermore, the local isometry $\tilde{p}$ induces a smooth SRF isometry $\tilde{p}: \tilde{M}/\tilde{\mathcal{F}}\rightarrow M/\mathcal{F}$ such that $d\tilde{p}(\tilde{H}_*)=H_*$, where $H$ and $\tilde{H}$ are the mean curvature vector fields of $(M,\mathcal{F})$ and $(\tilde{M},\tilde{\mathcal{F}}),$ respectively. 
\end{lemma}
\begin{proof}

See standard references such as \cite{Mol} regarding the lift of a foliation. With the lifted metric on $\tilde{M}$, the lifted foliation $\tilde{\mathcal{F}}=\tilde{p}^*(\mathcal{F})$ is a singular Riemannian foliation with possibly non-connected leaves. The covering map $\tilde{p}:\tilde{M}\rightarrow M$ is a foliated diffeomorphism that sends the leaves of $\tilde{\mathcal{F}}$ onto the leaves of $\mathcal{F}$. Furthermore, since $\tilde{M}$ is equipped with the lifted metric, $\tilde{p}$ is a local isometry, and hence $d\tilde{p}(\tilde{H})=H.$ To prove that $\tilde{p}$ induces a smooth SRF isometry on the leaf spaces, which we will also denote by $\tilde{p},$ we note that the covering map induces a metric space isometry on the quotients because the distance between leaves can be realized by a minimizing horizontal geodesic in the fundamental region. Since these lift to the cover, the corresponding map on the leaf spaces is an isometry on the quotient, with $d\tilde{p}(\tilde{H}_*)=H_*.$ 
\end{proof}

\begin{proof}[Proof of Corollary \ref{cor:spaceforms}.]
As noted above, any singular Riemannian foliation in spheres or Euclidean space is generalized isoparametric. Thus, by Theorem \ref{main}, it is enough to show that $\varphi$ preserves the projections of the basic mean curvature vector fields of $(M_1,\mathcal{F}_1)$ and $(M_2,\mathcal{F}_2).$ 

Suppose first that $\kappa=0,$ and $M_1=\mathbb{R}^{n_1}$ and $M_2=\mathbb{R}^{n_2}.$ Proposition 3.1 of \cite{AR2015} immediately yields the projections of the mean curvature vector fields are preserved. Now suppose $M_1=\mathbb{R}^{n_1}/\Gamma_1$ and $M_2=\mathbb{R}^{n_2}/\Gamma_2$ where $\Gamma_1$ and $\Gamma_2$ are finite subsets of the isometries of the appropriate Euclidean spaces, acting properly discontinuously.  Let $\tilde{p}_i: \mathbb{R}^{n_i}\rightarrow M_i$ be the usual covering maps, for $i=1, \,2$. By means of Lemma \ref{cor:coveringspace} applied to each covering map, $\varphi: M_1/\mathcal{F}_1\rightarrow M_2/\mathcal{F}_2$ lifts to a smooth SRF isometry $\tilde{\varphi} :\mathbb{R}^{n_1}/\tilde{\mathcal{F}}_1\rightarrow \mathbb{R}^{n_2}/\tilde{\mathcal{F}}_2$ by $\tilde{\varphi}=\tilde{p}_2^{-1}\circ \varphi \circ\tilde{p}_1.$ Furthermore, the map $\tilde{\varphi}$ preserves the mean curvature vector field if and only if $\varphi$ does, and the result follows.

Now suppose $\kappa>0$ and that $M_1=\mathbb{S}_r^{n_1}/\Gamma_1$ and $M_2=\mathbb{S}_r^{n_2}/\Gamma_2$ are space forms of the same positive constant sectional curvature $\kappa=1/r^2$, where $\Gamma_1$ and $\Gamma_2$ are finite subsets of the isometries of $\mathbb{S}^{n_1}_r$ and $\mathbb{S}^{n_2}_r$, respectively, acting properly discontinuously. If $\Gamma_1$ and $\Gamma_2$ are trivial, then the result follows immediately from Lemma \ref{lemmaA}. If they are not trivial, then by letting $\tilde{p}_i: \mathbb{S}^{n_i}_r\rightarrow M_i$ be the usual covering maps, for $i=1, \,2$, and applying Lemma \ref{cor:coveringspace} to each covering map as above, it follows that $\varphi: M_1/\mathcal{F}_1\rightarrow M_2/\mathcal{F}_2$ lifts to a smooth SRF isometry $\tilde{\varphi} :\mathbb{S}^{n_1}_r/\tilde{\mathcal{F}}_1\rightarrow \mathbb{S}^{n_2}_r/\tilde{\mathcal{F}}_2$ by $\tilde{\varphi}=\tilde{p}_2^{-1}\circ \varphi \circ\tilde{p}_1.$ Thus, by Lemma \ref{cor:coveringspace}, the map $\tilde{\varphi}$ preserves the mean curvature vector field if and only if $\varphi$ does. And now Lemma \ref{lemmaA} shows that $\tilde{\varphi}$ does indeed preserves the mean curvature vector fields, and thus we have the desired result for these space forms as well.

\end{proof}

\begin{rem}
\normalfont Proposition 3.1 of \cite{AR2015} is proved by using the characterization of the mean curvature vector field as the trace of the shape operator of the leaves, $S_{x}$, for $x\in V_p^\perp,$ the space of horizontal vectors to the foliation at the point $p$. Under the hypotheses and notation of this proposition, if there exists a $p_1\in M_1$ and an $x_1\in (V_1)^\perp_{p_1}$ such that all the eigenvalues of $S_{x_1}$ are non-zero, then $dim(M_1)=dim(M_2).$ This follows from the fact that if all the eigenvalues of $S_{x_1}$ are non-zero, the same is true for the eigenvalues of $S_{x_2}$. Further, the eigenvectors of $S_{x_1}$ that correspond to (non-zero) eigenvalues are in correspondence with the eigenvectors of $S_{x_2}$ that correspond to (non-zero) eigenvalues. These eigenvectors span the tangent space to the leaves of both $M_1$ and $M_2$. Since the leaf codimensions are the same by hypothesis, it follows that the dimensions of $M_1$ and $M_2$ must be the same. In particular, for quotients of spheres, such points $p_1$, as above, exist. Thus in the previous proof, for the case of $\kappa>0$, the dimensions of the spheres are equal. (See the alternative proof in Appendix for a more detailed calculation of the eigenvalues and their multiplicities in this case of spheres.)

\end{rem}
\end{subsection}

\begin{subsection}{The obstacle to isospectrality for isometric SRF}
In light of the proof, we now gain some insight into the lack of isospectrality {\it in general} when two leaf spaces are isometric. Suppose we have a smooth SRF isometry $\varphi: M_1/\mathcal{F} _1\rightarrow M_2/\mathcal{F}_2.$ Then, in the notation of the proof of Theorem \ref{main} (see equations \eqref{ub1} and \eqref{ub2} which hold on an open dense subset of the leaf spaces), we know that $\varphi^*\Delta_{B_2}=\Delta_{B_1}$ and $\varphi^*g_2=g_1$. However, the remaining ingredients in the expression for the Laplacians in \eqref{ub1} and \eqref{ub2}, are the mean curvature vector fields, $H_1$ and $H_2$. These consist of purely leaf-wise data, from Lemma 5.2 of \cite{AR2016b}, and thus are not related by an isometry of the leaf spaces. Hence, the two Laplacians could be different operators (with the same principal symbol). As such, there is no reason to expect isospectrality unless additional assumptions are added to ensure that the mean curvature vector fields correspond. Corollary \ref{cor:spaceforms} shows that this is possible in the case of constant curvature with the addition of the leaf-codimension preservation hypothesis.

Finally, we note that one can linearize singular Riemannian foliations via a transverse version of the exponential map to the infinitesimal foliation (see, for example, \cite{AR2015}) on the tangent space, which is a flat space form. Indeed, passing to the linearized foliation is a standard method in SRF geometry. However, this exponential map does not necessarily relate the mean curvature vectors on the original manifold to that on the tangent space, and thus applying Corollary \ref{cor:spaceforms} to the infinitesimal foliation is not necessarily useful in addressing questions about the basic spectra.

\end{subsection}
\end{section}

\begin{section}{Applications of the Main Result}\label{consequences}

\begin{subsection}{Applications to Isometric Group Actions and their Reductions}
There are a number of situations in which smooth SRF isometries between orbit spaces are known to exist. Many of these arise from metric space isometries that are known to be, under certain circumstances, smooth. For example, any metric space isometry of a leaf space that has an orbifold structure is smooth, Theorem 1.3 of \cite{AL2011}, or any metric space isometry between orbit spaces of dimension at most 3, Theorem 1.5 of \cite{AL2011}. These metric space isometries, in turn, come about because of standard reductions of the orbit spaces whereby one can replace the orbit space $M_1/G_1$ with another one, $M_2/G_2$, which typically arises from a simpler group action. Examples of such reductions include the principal isotropy reduction, the minimal reduction, and the effectivization of a group action.

\begin{cor}
Suppose $M$ is a compact manifold that admits isometric actions by closed, connected Lie groups $G_1$ and $G_2$, respectively. If the two actions are orbit equivalent (in the sense that for all $x\in M,$ $G_1\cdot x=G_2\cdot x$) then the $G_1$-invariant spectrum and the $G_2$-invariant spectrum are equal.
\end{cor}
\begin{proof}
Under the hypotheses above, the mean curvature vector fields over the regular regions of both actions are basic, and since the $G_1$-orbits are the same as the $G_2$-orbits, the leaf spaces are identical, and the orbit codimensions are the same for both actions. Hence, by Theorem 1.1 of \cite{AR2015}, the identity map is a smooth SRF isometry. The mean curvature vector fields are identical and the result follows from Theorem \ref{main}.
\end{proof}

As an application of the above, we have the following for Lie groups acting by isometries on a manifold $M$.
\begin{cor}
Let $H=\cap_{x\in M}G_x$, the intersection of all the isotropy subgroups. Then $K=G/H$ acts effectively on $M$ and the $G$-invariant spectrum and the $K$-invariant spectrum are identical.
\end{cor}
\begin{proof}
It is straightforward to check that $H$ is normal in $G$. Further, $K$ acts on $M$ because the elements of $H$ fix every point, and hence the action is well-defined, and effective (by construction). The reduction of the action to $M/K$ defines a metric space isometry $\varphi$ between $M/G$ and $M/K$ in a trivial way, since this reduction is an orbit equivalence. Thus, trivially, the dimensions and codimensions of the orbits are preserved.  By the previous corollary, the $G$-invariant spectra and the $K$ invariant spectra are the same.
\end{proof}

\end{subsection}

\begin{subsection}{Properties of SRF isometries}
We conclude this section with some remarks about smooth SRF isometries. In general, there are many different ways to represent a space $Q$ as the leaf space of a singular Riemannian foliation $(M,\mathcal{F})$, even for quotients $Q$ that have an orbifold structure. Indeed, in \cite{GL2015} the authors describe various singular Riemannian foliations which carry an orbifold structure on the leaf spaces, in the sense previously noted. In fact, they describe families of representations of isometric group actions that produce the same orbifold. However, in many cases, these orbifold quotients descend from singular Riemannian foliations that have variable leaf codimension, and thus do not meet the conditions of the previous corollaries that are sufficient to guarantee equality of the $\mathcal{F}$-basic and orbifold spectra.  

Further, SRF isometries generally do not preserve many features of group actions and singular Riemannian foliations such as isotropy type, or number of strata in the stratification by leaf dimension. To see this, we briefly revisit Example \ref{ex1}. In that example, a smooth SRF isometry exists between the orbifold $\mathcal{O}=[0,\pi]$ and $S^2/SO(2)$. However, $\mathcal{O}$ is an orbifold, and hence arises from a regular Riemannian foliation via the frame bundle representation (among others), and thus has only one stratum consisting of leaves of codimension 1. On the other hand, $S^2/SO(2)$ has a stratification consisting of two strata--a singular stratum consisting of the two poles of codimension 2, and the complementary set, whose leaves have codimension 1. In addition, $\mathcal{O}$ has $\mathbb{Z}_2$ isotropy at the endpoints, while $S^2/SO(2)$ has $SO(2)$ isotropy at the corresponding points. Thus, smooth SRF isometries do not preserve either the number of strata in a stratification, the codimension of the leaves over singular strata, or the isotropy type of corresponding points. These isometries do preserve topological features of the leaf space, such as dimension of the leaf space, the presence of (topological) boundary, and some features of the images of stratified sets in the ambient space $M.$

We recall first the definition of quotient codimension. For a leaf space $\pi:M\rightarrow M/\mathcal{F}$, if $\Sigma$ is a subset of $M$ saturated by leaves, then the quotient codimension of $\Sigma$ is defined to be:
\begin{equation}
qcodim(\Sigma)=dim\bigl(\pi(M_{reg})\bigr)-dim\bigl(\pi(\Sigma)\bigr).
\end{equation}

The following demonstrates that a smooth SRF isometry preserves the quotient codimension of saturated sets. 

\begin{prop}\label{codim}
Let $\Sigma^1_k$ denote the union of strata in $M_1$ with quotient codimension $k$, and let $\overline{\Sigma^1_k} :=\pi_1(\Sigma^1_k),$ and similarly define $\overline{\Sigma^2_k}$. If $\varphi$ is a smooth SRF isometry between leaf spaces as above, then $\varphi(\overline{\Sigma^1_k})=\overline{\Sigma^2_k}$ for each $k$ for which $\Sigma^1_k\not=\emptyset$.
\end{prop}

\begin{proof}
$M_1$ and $M_2$ are each stratified by a finite collection of sets $S_i,$ respectively $T_j$ where $S_i$ is the union of leaves of dimension $i$, and that the dimension of the leaves is lower semi-continuous, \cite{Mol}, and similarly for $T_j$. Here $i$ takes values from $\{0, i_1\dots,i_{max}\}$, where $i_{max}$ is the maximum dimension of the orbits. Hence, $M_1=\bigcup_i S_i$ and $\overline{S_i}\subset \bigcup_{k\le i} S_k,$ and similarly for $M_2.$ In fact, $S_i$ is open and dense in $\bigcup_{k\le i} S_k.$ Note, these stratifications may be a little finer than the stratifications by quotient codimension. By assumption, $\varphi$ is a homeomorphism. It follows easily that when $k=0$, $\overline{\Sigma^1_0}$ contains $\pi(S_{i_{max}})$ and thus is dense with non-empty interior in $M_1/\mathcal{F}_1$ and $\varphi(\overline{\Sigma^1_0})=\overline{\Sigma^2_0}.$ Note: it follows that $dim\bigl(\pi_1((M_1)_{reg})\bigr)=dim\bigl(\pi_2((M_2)_{reg})\bigr).$ Hence, $\varphi$ restricts to a homeomorphism on the complements: $\varphi((\overline{\Sigma^1_0})^c)=(\overline{\Sigma^2_0})^c.$ Let $k_1$ denote the first non-zero quotient codimension of $M_1$. Then $\overline{\Sigma^1_{k_1}}$ contains $\pi(S_{i_1})$ and thus is dense with non-empty interior in $(\overline{\Sigma^1_0})^c$ and so $\varphi(\overline{\Sigma^1_{k_1}})$ has the same property in $(\overline{\Sigma^2_0})^c.$  Further, its dimension is equal to $k_1$. It follows that $k_1$ is also the first non-zero quotient dimension in $M_2$, and thus $\varphi((\overline{\Sigma^1_{k_1}})^c)=(\overline{\Sigma^2_{k_1}})^c,$ where the complements are taken in $(\overline{\Sigma^1_0})^c$ and $(\overline{\Sigma^2_0})^c.$  This argument can be repeated finitely many times to show the conclusion.
\end{proof}
 
Finally, we note that in \cite{AS2017}, there are examples of isospectral leaf spaces that have strata of different quotient codimensions, and thus, this situation is not audible.

\end{subsection}
\end{section}
\begin{section}{Appendix}
Here, we include an alternate proof of Corollary \ref{cor:spaceforms} using the characterization of the mean curvature vector field in terms of the trace of the shape operator of the leaves. In particular, we show that the projections of the mean curvature vector fields are preserved using Jacobi fields and special properties of foliations on spheres to calculate the eigenvalues and their multiplicities of the shape operators. This proof starts by proving the result in the case of a regular Riemannian foliation in the first lemma below, before generalizing to the singular case. This is essentially just a modification of Proposition 4.1.1 of \cite{GW2009}.

\begin{lemma}\label{l:regular}
If $M_1$ and $M_2$ are space forms of the same non-negative curvature $\kappa$ that admit regular Riemannian foliations $\mathcal{F}_1$ and $\mathcal{F}_2,$ and $\varphi:M_1/\mathcal{F}_1 \rightarrow M_2/\mathcal{F}_2$ is a smooth SRF isometry, then $d\varphi(H_{1*})=H_{2*}$.
\end{lemma}

\begin{proof}
Let $p_1\in M_1$ and $p_2\in M_2$. Assume further that $\bar{p}_1\in M_1/\mathcal{F}_1$ and $\bar{p}_2=\varphi(\bar{p}_1)\in M_2/\mathcal{F}_2$ denote their images in the respective quotient spaces--i.e.~$\pi_1(p_1)=\bar{p}_1$ where $\pi_1:M_1\rightarrow M_1/\mathcal{F}_1$ and similarly for $\bar{p}_2$. Let $x_i\in (V_i^\perp)_{p_i}$ denote horizontal vectors for $i=1,\,2$ and let $\bar{x}_i =d\pi_i(x_i)$ denote the images of these vectors in the quotient space. Suppose further that $\bar{x}_2:=d\varphi(\bar{x}_1).$ Let $\gamma_i(t)=exp_{p_i}(tx_i)$ for $i=1,\,2$. Finally, let $S_{x_i}$ denote the shape operators at $x_i$ for $M_i$, $i=1,\,2$.

Let $u_1\in (V_1)_{p_1}$ be the eigenvector with eigenvalue $\lambda$ for $S_{x_1}$. Recall from \cite{GW2009} that a Jacobi field $J$ defined along a curve $c$ is said to be {\it projectable} if and only if it satisfies $J'^{\bs{v}}=-S_{\dot{c}}J^{\bs{v}}-A_{\dot{c}}J^{\bs{h}}$. (Here, the superscripts denote the vertical and horizontal components of the vector fields.) And if $J$ is vertical and projectable, $J'^{\bs{v}}=-S_{\dot{c}}J.$

Now consider the Jacobi field $J_1(t)$ along $\gamma_1(t)$ defined by the following initial value problem: $J_1''(t)+\kappa J_1(t)=0,$ $J_1(0)=u_1$ and $J_1'(0)=-S_{x_1}u_1=-\lambda u_1$. Since $u_1$ is vertical, $J_1$ is vertical. The initial condition on the first derivative implies that $J_1$ is projectable, a fact which we will use shortly.

Returning to the initial value problem, we observe that $J_1(t)=f(t)E(t)$ where $E$ is the parallel field along $\gamma_1$ with $E(0)=u_1$, and $f$ satisfies the initial value problem $f''+\kappa f=0,$ $f(0)=1,$ $f'(0)=-\lambda$. By solving, we see that when $\kappa=0$, $f(t)=1-\lambda t$, and when $\kappa>0$ we have $f(t)= \cos(t\sqrt{\kappa})-\frac{\lambda}{\sqrt{\kappa}}\sin(t\sqrt{\kappa}).$

Assume for the moment that $\lambda\not=0$ if $\kappa=0$. Then $J_1(t_0)=0$ for some $t_0\in \mathbb{R}$. If $\kappa=0,$  $f(t_0)=0$ implies that $\lambda=1/t_0.$ If $\kappa>0$, then $f(t_0)=0$ imples that $\lambda=\frac{\sqrt{\kappa}}{\tan(t_0\sqrt{\kappa})}.$

Since $J_1$ is projectable, by Theorem 1.6.1 of \cite{GW2009}, we have that $d\pi_1(J_1)$ is a Jacobi field along $\pi_1(\gamma_1)$.
It follows easily that $d\varphi(d\pi_1(J_1))$ is a Jacobi field along $\varphi(\pi_1(\gamma_1(t)))=\pi_2(\gamma_2(t))$. By Lemma 1.6.1 of \cite{GW2009}, there exists a unique Jacobi $J_2$ along $\gamma_2$ with $d\pi_2(J_2)=d\varphi(d\pi_1(J_1)),$ $J_2'^{\bs{v}}(0)=-S_{x_2}J_2(0),$ and $J_2(t_0)=0$. In particular, $J_2$ must be vertical, because $J_1$ is vertical. Since $J_2(t_0)=0$, and $J_2$ is vertical, $J_2=fE_2$ for some parallel field $E_2$. Hence, $J_2'(0)=-\lambda J_2(0),$ and thus $J_2(0)$ is an eigenvector with eigenvalue $\lambda$ of $S_{x_2}$.

Now consider multiplicities. Suppose $\lambda$ is an eigenvalue of $S_{x_1}$ with multiplicity $k$. Let $E_1(\lambda)$ denote the $\lambda$ eigenspace of $S_{x_1}.$ As noted in the proof of Proposition 4.1.1 of \cite{GW2009}, for any Jacobi field $\bar{J}_1$ along the projected geodesic $\pi_1\circ \gamma_1$ with $\bar{J}_1(0)=\bar{J}_1(t_0)=0$ then, by Lemma 1.6.1 of \cite{GW2009}, there exists a unique projectable Jacobi field $J_1$ along $\gamma_1$with the $J_1(0)\in E_1(\lambda)$. This implies that $t_0$ is a conjugate point of $\pi_1\circ\gamma_1$ of multiplicity (as a conjugate point) of at most $k$. 

On the other hand, if $J_1$ is a projectable Jacobi field along $\gamma_1$ with $J_1(0)=-S_{x_1}J_1(0),$ with $J_1(t_0)=0,$ $J(0)\not=0$ then $d\pi_1(J_1)$ is a Jacobi field that vanishes at $t=0$ (because $J_1(0)\in E_1(\lambda)\subset V_1$) and at $t=t_0.$ We may assume that $d\pi_1(J_1)$ is non-trivial, since otherwise $J_1$ is vertical, and thus a holonomy field (See Definition 1.4.3, of \cite{GW2009}). Such fields vanish identically at any point; thus, the condition that $J_1(t_0)=0,$ implies that $J_1$ is trivial, a contradiction. Thus, $t_0$ is a conjugate point of multiplicity exactly $k$ for $\pi_1\circ\gamma_1$.

By an identical argument, if $\lambda$ is an eigenvalue of $S_{x_2}$ with multiplicity $k'$, then $t_0$ is a conjugate point of multiplicity exactly $k'$ for $\pi_2\circ\gamma_2$.

But, $\varphi\circ\pi_1\circ\gamma_1=\pi_2\circ\gamma_2$, and $d\varphi(d\pi_1(J_1))=d\pi_2(J_2)$. Hence, $J_1(t_0)=0$ if and only if $J_2(t_0)=0$, so conjugate points are preserved, and therefore $k=k'$, and thus the eigenvalues of $S_{x_1}$ are the same as those of $S_{x_2}$, including multiplicities. Hence, $d\varphi(H_{1*})(\bar{p}_1)=H_{2*}(\bar{p}_2)$.

If $\kappa=0$, some care must be taken with the zero eigenvalues. Suppose now that $\kappa=0$ and $\lambda=0$, then $J_1(t)\not=0$ for finite $t$. In this case, we can take $t_0=\infty$ and $\lambda=1/t_0=0.$ By the argument above, $J_2(t)\not=0$ for all $t$, and hence $J_2(0)$ will be an eigenvector of $S_{x_2}$ with eigenvalue zero.

We have now demonstrated that the non-zero eigenvalues of $S_{x_1}$ are the same as those of $S_{x_2}$, counting with multiplicities. Hence, their traces are the same, and it follows that $d\varphi(H_{1*})(\bar{p}_1)=H_{2*}(\bar{p}_2),$ and thus $d\varphi(H_{1*})=H_{2*}.$
\end{proof}

\begin{proof}[Alternative proof of Corollary \ref{cor:spaceforms}.]
Let $p_1\in M_1$ be a point belonging to the regular region, and let $x_1\in (V_1^\perp)_{p_1}$, and let $X$ be a horizontal vector field with $X_{p_1}=x_1$.  Define $p_2$ and $x_2$ similarly with $p_2=\varphi(p_1)$ and $x_2=d\varphi(X_{p_1}).$ Note that $p_2$ necessarily lies in the regular region as well because of the leaf codimension preservation condition. Let $u_1\in (V_1)_{p_1}$ be the eigenvector with eigenvalue $\lambda$ for $S_{x_1}$, as in the proof of Lemma \ref{l:regular}, and let $\gamma_1(t)=exp_{p_1}(tx_1).$ Now consider the Jacobi field $J_1(t)$ along $\gamma_1(t)$ defined by the following initial value problem: $J_1''(t)+\kappa J_1(t)=0,$ $J_1(0)=u_1$ and $J_1'(0)=-S_{x_1}u_1=-\lambda u_1$. As above, $J_1(t)$ will vanish at some $t_0\in \mathbb{R}$ unless $\kappa=0$. For the moment, we will assume that such a $t_0<\infty$ exists. If $\gamma_1(t_0)$ belongs to the regular region of $M_1$, then we may use the argument of the previous lemma to show that $d\varphi(H_{1*})(\bar{p}_1)=H_{2*}(\bar{p}_2)$.

If $\gamma_1(t_0)$ belongs to the singular strata of $M_1$ such that $\pi_1\circ\gamma_1(t)$ lies in the orbifold part of $M_1/\mathcal{F}_1$, denoted by $(M_1/\mathcal{F}_1)_{orb}$, then we will show that the non-zero eigenvalues of $S_{x_1}$ are the same as the non-zero eigenvalues of $S_{x_2}$ when $\kappa=0,$ and all eigenvalues of $S_{x_1}$ are the same as the eigenvalues of $S_{x_2}$ when $\kappa>0.$

Note that the complement of $(M_1/\mathcal{F}_1)_{orb}$ has codimension at least two, thus almost every projected horizontal geodesic stays in $(M_1/\mathcal{F}_1)_{orb}$ for all time. Further $\varphi((M_1/\mathcal{F}_1)_{orb})=(M_2/\mathcal{F}_2)_{orb},$ and $\varphi\circ\pi_1\circ\gamma_1=\pi_2\circ\gamma_2.$ From this we will conclude that the trace of $S_{x_1}$ equals the trace of $S_{x_2}$ whenever $\pi_1\circ \gamma_1$ is contained in $(M_1/\mathcal{F}_1)_{orb},$ which is open and dense. By continuity of the mean curvature form, we will then conclude that $d\varphi(H_{1*})=H_{2*}.$

We now take cases, supposing first $\kappa>0$, then $\kappa=0$.

If $\kappa>0$, then we may use the following argument from the thesis of M. Radeschi, \cite{Radeschi2012}. We will first suppose that $M_1$ is a round $n_1$-sphere of curvature $\kappa$, and $M_2$ is similarly a round $n_2$-sphere of curvature $\kappa$, each admitting singular Riemannian foliations $\mathcal{F}_1$ and $\mathcal{F}_2,$ respectively. At the conclusion of this case, we will show how to generalize the result from spheres to positive curvature space forms.

From \cite{Mol}, Lemma 6.1, we know that if $L_s$ is a singular leaf through $p_0:=\gamma_1(t_0)$, and $\sigma$ is a horizontal geodesic starting at $p_0$ that leaves the stratum containing $p_0$, then for some small $\varepsilon$ there is a neighborhood $U_{\sigma(\varepsilon)}\subset L_{\sigma(\varepsilon)},$ (contained in the regular region), and a neighborhood $U_{p_0}\subset L_s$ such that the closest-point map $\Pi: U_{\sigma(\varepsilon)}\rightarrow U_{p_0}$ is a submersion with non-trivial kernel and $d \Pi_{\sigma{\epsilon}} v = J_v(0)$ where $J_v$ is the holonomy Jacobi field along $\sigma$ such that $J_v{\epsilon}=v.$
We can then prove the following claim (generalizing Lemma 3.0.6 of \cite{Radeschi2012}) which we include for the sake of completeness of the exposition:  Let $(M,\mathcal{F})$ be a singular Riemannian foliation on a sphere $M$ of curvature $\kappa$, and suppose $L_s$ is a singular leaf, $\sigma$ is a horizontal geodesic starting at $p_0\in L_s$, and $\Pi: U_{\sigma(\varepsilon)}\rightarrow U_{p_0}$ the local submersion, then 
\begin{equation}
Ker\, d\,\Pi_{\sigma(\varepsilon)}=\Bigl\{ v\in V_{\sigma(\varepsilon)}\,|\,A^*_xv=0,\, S_xv=\frac{\sqrt{\kappa}}{\tan(\varepsilon\sqrt{\kappa})}v\Bigr\}
\end{equation}
where $A^*$ is the adjoint of the O'Neill tensor, $A,$ and $x=-\sigma'(\varepsilon).$ Observe that $Ker\, d\,\Pi_{\sigma(\varepsilon)}$ is a subspace of dimension $dim(L_{\sigma(\varepsilon)})-dim(L_s)$.

This claim is proved as follows. Let $v\in Ker\, d\Pi_{\sigma(\varepsilon)}$ and let $J(t)$ be the unique holonomy Jacobi field along $\tilde{\sigma}(t):=\sigma(\varepsilon-t)$ such that $J(0)=v$. Then, since $J(t)$ is a holonomy Jacobi field, $J'(0)= -A^*_xv-S_xv$ with $x=\tilde{\sigma}'(0)$. We may assume $J(t)=\cos(\sqrt{\kappa}t)E_1(t)+\frac{\sin(\sqrt{\kappa}t)}{\sqrt{\kappa}}E_2(t)$ where $E_1(t)$ and $E_2(t)$ are parallel fields, $E_1(0)=v$, $E_2(0)=-A^*_xv-S_xv$ since $M$ is a sphere of curvature $\kappa$. Noting that $\langle E_i(t),E_j(t)\rangle$ for $i, j\in {1,2}$ are constant, since $E_1(t)$ and $E_2(t)$ are parallel fields, one easily calculates that $\|J(t)\|^2$ goes to zero precisely when $A^*_xv=0$ and $S_xv=\frac{\sqrt{\kappa}}{\tan(\varepsilon\sqrt{\kappa})}v,$ proving the claim.

Let $E_\lambda(x_1)$ and $E_\lambda(x_2)$ denote the $\lambda$ eigenspaces of $S_{x_1}$ on $M_1$ and $S_{x_2}$ on $M_2,$ respectively. We wish to show that $E_\lambda(x_1)$ and $E_\lambda(x_2)$ have the same dimension. By the previous claim taking $\varepsilon=t_0$, there are subspaces $K_{p_1}:=E_\lambda(x_1)\cap ker A^*_{x_1}$ and $K_{p_2}:=E_\lambda(x_2)\cap ker A^*_{x_2}$. These subspaces have the same dimension precisely when $codim(L_{\gamma_1(t_0)})=codim(L_{\gamma_2(t_0)}).$ It is now enough to show that the dimensions of the quotient spaces $dim(E_\lambda(x_1)/K_{p_1})=dim(E_\lambda(x_2)/K_{p_2})$.

We do so by constructing a bijection between the quotient spaces, as in the proof of Proposition 3.0.7 of \cite{Radeschi2012}, which we include for the sake of completeness.

We begin by showing that there is a well-defined injective map
\begin{equation}\label{e:EmodK}
E_\lambda(x_1)/K_{p_1}\rightarrow E_\lambda(x_2)/K_{p_2}.
\end{equation}
The existence of a corresponding map in the other direction will follow from reversing the roles of $p_1$ and $p_2$ and using $\varphi^{-1}$ in place of $\varphi$. It is helpful to define the following. Let
\begin{eqnarray}
\mathcal{K}_1&:=&\{J_1\,|\,J_1 \text{ is a Jacobi field along }\gamma_1 \text{ with } J_1(t_0)=0\}\\
\mathcal{K}_2&:=&\{J_2\,|\,J_2 \text{ is a Jacobi field along }\gamma_2 \text{ with } J_2(t_0)=0\}.
\end{eqnarray}
Let $ev_0$ be the evaluation map at $t=0$. From the argument above, we have $ev_0(\mathcal{K}_1)=K_{p_1}$ and $ev_0(\mathcal{K}_2)=K_{p_2}.$

Now let $[v]\in E_\lambda(x_1)/K_{p_1},$ and let $v\in E_\lambda(x_1)\subset V_{p_1}$ be any representative of $[v]$. Let $J_v(t)$ be the projectable Jacobi field along $\gamma_1$ with $J_v(0)=v$, $J'_v(0)=-S_{x_1}(v)=-\lambda v.$ Consider an interval of the form $I=(t_0-\varepsilon, t_0)$ for small $\varepsilon>0$ such that $\gamma_1|_I$ is on the regular part of the foliation. On this interval $J_v$ is a projectable Jacobi field over $\gamma_1$ and we may use Theorem 1.6.1 of \cite{GW2009} to conclude that $d\pi_1(J_v(t))$ is a Jacobi field along the projected geodesic $\pi_1\circ \gamma_1$ that vanishes at $t=0$. Furthermore, $\lim_{t\to t_0}\|J_v(t)\|=0$, so it vanishes at $t_0$ as well. Note that $d\pi_1(J_v(t))$ is defined independently of the choice of representative $v$.

Consider $\gamma_2$ a geodesic in $M_2$ such that $\varphi\circ\pi_1\circ\gamma_1=\pi_2\circ\gamma_2$, restricted to $I=(t_0-\varepsilon,t_0)$. We note that $d\varphi(d\pi_1(J_v))$ is a Jacobi field along $\pi_2\circ\gamma_2$, and $\gamma_2|_I$ lies in the regular region of $M_2$. Thus, we may use Lemma 1.6.1 of \cite{GW2009} to find the projectable Jacobi field $\tilde{J}_v(t)$ that projects to $d\pi_2(\tilde{J}_v)=d\varphi(d\pi_1(J_v))$, and such that $\tilde{J}_v(t_0)=0.$ This $\tilde{J}_v$ is uniquely defined up to a Jacobi field in $\mathcal{K}_2$, and $\tilde{J}_v(0)$ is an eigenvector of $S_{x_2}$ with eigenvalue $\lambda=\frac{\sqrt{\kappa}}{\tan(\sqrt{\kappa}t_0)},$ which is well-defined up to an element of $K_{p_2}=ev_0(\mathcal{K}_2)$. This defines the map $[v]\mapsto[\tilde{J}_v(0)]$ in \eqref{e:EmodK}.

Note by assumption the codimensions of the corresponding leaves are the same. By the above argument, we see that the corresponding dimensions of the eigenspaces are the same. Since the eigenvectors span the tangent spaces of the leaves, we see that $n_1=n_2.$
Now suppose that $M_1=S^{n_1}/\Gamma_1$ and $M_2=S^{n_2}/\Gamma_2$ are space forms of the same positive constant curvature $\kappa$, where $\Gamma_1$ and $\Gamma_2$ are finite subsets of the isometries of $S^{n_1}$ and $S^{n_2}$, respectively, acting properly discontinuously.  Let $\tilde{p}_i: S^{n_i}\rightarrow M_i$ be the usual covering maps, for $i=1, \,2$. By means of Lemma \ref{cor:coveringspace} applied to each covering map, $\varphi: M_1/\mathcal{F}_1\rightarrow M_2/\mathcal{F}_2$ lifts to a smooth SRF isometry $\tilde{\varphi} :S^{n_1}/\tilde{\mathcal{F}}_1\rightarrow S^{n_2}/\tilde{\mathcal{F}}_2$ by $\tilde{\varphi}=\tilde{p}_2^{-1}\circ \varphi \circ\tilde{p}_1,$ and we may conclude as in the preceding paragraph that $n_1=n_2.$ By Lemma \ref{cor:coveringspace}, the map $\tilde{\varphi}$ preserves the mean curvature vector field if and only if $\varphi$ does. By the previously covered case for spheres, $\tilde{\varphi}$ preserves mean curvature, and we have the desired result for these space forms as well.

If $\kappa=0$, then all the non-zero eigenvalues of $S_{x_1}$ are the same as the non-zero eigenvalues of $S_{x_2}$ and have the same multiplicities for $M_1=\mathbb{R}^{n_1}$ and $M_2=\mathbb{R}^{n_2}$ by Proposition 3.1 of \cite{AR2015}. 

Suppose now $M_1=\mathbb{R}^{n_1}/\Gamma_1$ and $M_2=\mathbb{R}^{n_2}/\Gamma_2$ where $\Gamma_1$ and $\Gamma_2$ are finite subsets of the isometries of the appropriate Euclidean spaces, acting properly discontinuously. The covering map argument of the $\kappa>0$ case can then be modified to show the result follows for compact curvature zero space forms as well. \end{proof}
\end{section}
\section*{References}

\bibliographystyle{alpha} 

\end{document}